\documentclass[leqno]{article}
\usepackage{amsmath}
\usepackage{amssymb}
\usepackage{amsthm}
\def\C{{\mathbb{C}}}

\def\N{{\mathbb{N}}}
\def\Z{{\mathbb{Z}}}

\def\ppmod{\mkern-16mu \pmod}
\def\RE{\operatorname{Re}}

\theoremstyle{plain}
\newtheorem{theorem}{Theorem}[section]
\newtheorem{corollary}{Corollary}[section]
\newtheorem{lemma}{Lemma}[section]

\theoremstyle{definition}

\newtheorem{example}{Example}[section]

\theoremstyle{remark}
\newtheorem{remark}{Remark}[section]

\begin{document}

\title{An analogue of Ramanujan's sum with respect to\\ regular integers (mod $r$)}
\author{Pentti Haukkanen and L\'aszl\'o T\'oth}
%\date
\maketitle

\medskip
\noindent
{\bf Abstract}\ An integer $a$  is said to be regular (mod $r$)
if there exists an integer $x$ such that $a^2x\equiv a\pmod{r}$. 
In this paper we introduce 
an analogue of Ramanujan's sum with respect to regular integers (mod $r$)
and show that this analogue possesses properties similar to 
those of the usual Ramanujan's sum. 



\section{Introduction}

An element $a$ in a ring $R$ is said to be regular (following von Neumann)
if there exists  $x\in R$ such that $axa=a$.
An integer $a$  is said to be regular (mod $r$)
if there exists an integer $x$ such that $a^2x\equiv a\pmod{r}$. 
A regular integer (mod $r$) is regular in the ring $\Z_r$ in the sense of von Neumann. 

An integer $a$ is invertible (mod $r$) if $(a, r)=1$. 
It is clear that each invertible integer (mod $r$) is regular (mod $r$). 
Euler's function $\phi(r)$ counts the number of invertible integers (mod $r$). 
The function $\varrho(r)$ counts the number of regular integers (mod $r$) between $1$ and $r$ and 
thus $\varrho(r)$ is an analogue of Euler's $\phi$-function with respect to  regular integers (mod $r$). 
Regular integers (mod $r$) and the function $\varrho(r)$ have been studied, 
for example, in \cite{AO,Tot2008,To}. 

Ramanujan's \cite{R} sum $c_r(n)$ is defined as
\begin{equation}
c_r(n)=\sum_{\substack{a \ppmod r\\ (a, r)=1}}\exp(2\pi ian/r). 
\end{equation}
In this paper we introduce 
an analogue of Ramanujan's sum with respect to regular integers (mod $r$)
and show that this analogue possesses properties similar to 
those of the usual Ramanujan's sum. 
We study convolutional expressions, even functions (mod $r$), discrete Fourier transform, 
identities, multiplicativity, mean values and Dirichlet series. 
Throughout the paper we compare the properties of this analogue of Ramanujan's sum with 
the properties of the usual Ramanujan's sum. 
Analogues of Ramanujan's sum with respect to regular integers (mod $r$) have not hitherto been studied 
in the literature. 

%%%%%%%%%%%%%%%%%%%%%%%%%%%%%%%%%%%%%%%%%%%%%%%%%%%

\section{Definition and convolutional expressions}

We define the analogue of Ramanujan's sum with respect regular integers (mod $r$) as
\begin{equation}
\overline{c}_r(n) 
= \sum_{
    \substack{a \ppmod r\\
              a \text{ regular}\ppmod r}}
    \exp(2 \pi i a n / r), 
\end{equation}
where $n\in\Z$ and $r\in\N$. 

The usual Ramanujan's sum can be written as 
\begin{equation}
c_r(n)=\sum_{d \mid (n,r)} d\mu(r/d),  
\end{equation}
which may be considered a convolutional expression of 
$c_r(n)$ with respect to $n$ or $r$.
Convolutional expressions of 
$\overline{c}_r(n)$ with respect to $n$ and $r$ are presented in 
Theorems \ref{th:conv-n} and \ref{th:conv-r}. 
In these expressions we need the concept of a unitary divisor. 
A divisor $d$ of $n$ is said to be a unitary divisor of $n$ 
(written as $d\| n$) if $(d, n/d)=1$. 
The unitary convolution of arithmetical functions $f$ and $g$ is defined as 
\begin{equation}
(f\oplus g)(n)=\sum_{d\| n} f(d) g(n/d). 
\end{equation}
For material on unitary convolution we refer to \cite{Cohen60,M,SC,Si}. 
We will use that $a$ is regular $\pmod r$ if and only if $(a, r)\Vert r$.

\medskip

Let $r\in\N$ be fixed. 
Let $g_{r}$ denote the characteristic function of the unitary divisors of $r$, 
that is, $g_{r}(n)=1$ if $n\| r$,
and $g_{r}(n)=0$ otherwise.
Then $g_{r}(n)$ is multiplicative in $n$.
Let  $\overline{\mu}_{r}$ denote the function defined by 
\begin{equation}\label{eq:mu-r}
(\overline{\mu}_{r} * 1)(n) = g_{r}(n), 
\end{equation}
where $*$ is the Dirichlet convolution and $1(n)=1$ for all $n\in\N$.
Then $\overline{\mu}_{r}(n)$ is multiplicative in $n$ given as follows:\\
(i)\ If $p\| r$, then $\overline{\mu}_{r}(p)=0$,
$\overline{\mu}_{r}(p^2)=-1$, $\overline{\mu}_{r}(p^j)=0 $ for $j \geq 3$.\\
(ii)\ If $p^a\| r$ with $a \geq 2$, then $\overline{\mu}_{r}(p)=-1$, 
$\overline{\mu}_{r}(p^a)=1$, $\overline{\mu}_{r}(p^{a+1})=-1$, $\overline{\mu}_{r}(p^j) = 0$    
for $j \neq 0,1,a,a+1$.\\
(iii)\ If $p\nmid r$, then 
$\overline{\mu}_{r}(p)=-1$, $\overline{\mu}_{r}(p^j)=0$ for $j\geq 2$.\\
In addition, $\overline{\mu}_{r}(1)=1$.  
 
\begin{theorem}\label{th:conv-n}
\begin{equation*}
\overline{c}_r(n)  = \sum_{d | (n, r)} d \overline{\mu}_{r}(r/d).  
\end{equation*}
\end{theorem}

\begin{proof}
Since $a$ is regular $\pmod r$ if and only if $(a, r)\Vert r$, see \cite{Tot2008}, we have 
\begin{eqnarray*}
\overline{c}_r(n)
&=& \sum_{a=1}^r \exp(2 \pi ian/r)g_{r}((a, r)) \\
&=& \sum_{a=1}^r \exp(2 \pi ian/r)\sum_{d\mid (a,r)}\overline{\mu}_{r}(d) \\
&=& \sum_{d\mid r}\overline{\mu}_{r}(d)\sum_{a=1\atop d\mid a}^r \exp(2 \pi ian/r) \\
&=& \sum_{d\mid r}\overline{\mu}_{r}(d)\sum_{m=1}^{r/d} \exp(2 \pi imdn/r) \\
&=& \sum_{d\mid r}\overline{\mu}_{r}(r/d)\sum_{m=1}^{d} \exp(2 \pi imn/d) \\
&=& \sum_{d\mid r\atop d\mid n}\overline{\mu}_{r}(r/d)d.
\end{eqnarray*}
\end{proof}

\begin{remark}\label{re:conv-n}
Let $\eta_r$ denote the arithmetical function defined as $\eta_r(n)=n$ if $n\mid r$, 
and $\eta_r(n)=0$ otherwise. 
Theorem \ref{th:conv-n} can be written as 
\begin{equation*}
\overline{c}_r(n)  =  [\eta_r(\cdot) \overline{\mu}_{r}\big(r/(\cdot)\big) * 1(\cdot)] (n).    
\end{equation*}
\end{remark}

\begin{remark}
Theorem \ref{th:conv-n} gives a convolutional expression of 
$\overline{c}_r(n)$ with respect to the variable $n$. 
\end{remark} 

\begin{corollary}\label{th:conv-n-sum} For $n\in\N$, 
\begin{equation*}
\sum_{d\mid n}\overline{c}_r(d)  = \sum_{d\mid (n, r)} d \overline{\mu}_{r}(r/d)\tau(n/d),   
\end{equation*}
where $\tau(n)$ is the number of divisors of $n$. 
\end{corollary}

\begin{proof}
Using Theorem \ref{th:conv-n} we obtain 
\begin{equation*}
\sum_{d\mid n}\overline{c}_r(d)  = (\overline{c}_r\ast 1)(n) 
=  [\eta_r(\cdot) \overline{\mu}_{r}\big(r/(\cdot)\big) * 1(\cdot) * 1(\cdot)] (n). 
\end{equation*}
Since $1\ast 1=\tau$, we obtain Corollary \ref{th:conv-n-sum}. 
\end{proof}

\begin{remark}
For Ramanujan's sum we have 
\begin{equation*}
\sum_{d\mid n} c_r(d)  = \sum_{d\mid (n, r)} d \mu(r/d) \tau(n/d).  
\end{equation*}
\end{remark}

We later need the following lemma involving the function $\overline{\mu}_{r}$. 

\begin{lemma}\label{le:mu-r}
For any arithmetical function $f$, 
\begin{equation*}
\sum_{d\mid r} f(d) \overline{\mu}_r(r/d)=[(f*\mu)\oplus 1](r). 
\end{equation*}
\end{lemma}

\begin{proof}
We have  
\begin{eqnarray*}
\sum_{d\mid r} f(d) \overline{\mu}_r(r/d)
&=& \sum_{de=r} f(d) \overline{\mu}_r(e)
=\sum_{de=r} f(d) \sum_{ab=e}\mu(a) g_r(b)\\
&=&\sum_{dab=r} f(d)\mu(a) g_r(b)
=\sum_{b\mid r} g_r(b) (f\ast\mu)(r/b)\\
&=&\sum_{b\| r} (f\ast\mu)(r/b)
=[(f\ast\mu)\oplus 1](r). 
\end{eqnarray*}
\end{proof}

%%%%%%%%%%%%%%%%%%%%%%%%%%%%%%%%%%%%%%%%%%%%%%%%%%%

\section{Even functions (mod $r$)}\label{sec:even}

Let $r\in\N$ be fixed. 
A function $f\colon\Z \to \C$ is said to be $r$-periodic or periodic (mod $r$)  if
$f(n)=f(n+r)$ for all $n\in \Z$. 
A function $f\colon\Z \to \C$ is said to be $r$-even or even (mod $r$) if
$f(n)=f((n, r))$ for all $n\in \Z$. 
Each $r$-even function is $r$-periodic.
Ramanujan's sum $c_r(n)$ is an example of an $r$-even function. 
It is easy to see on the basis of Theorem \ref{th:conv-n} that 
$\overline{c}_r(n)$ is also an $r$-even function. 
For material on $r$-periodic and $r$-even functions we refer to 
\cite{A,C,H2001,H2007,M,Sa,S2008,SS,Si,Sp,T}.

The discrete Fourier transform (DFT) of an $r$-periodic function $f$ is the function
$\widehat{f}\colon \Z \to \C$ defined by
\begin{equation} \label{DFT}
\widehat{f}(n)= \sum_{a=1}^r f(a) \exp(-2\pi ian/r).
\end{equation}
Each $r$-even function $f$ can be represented as
\begin{equation} \label{even_func}
f(n)= \frac1{r} \sum_{d\mid r} \widehat{f}(r/d) c_d(n),  
\end{equation}
see \cite{H2007}. 
The Cauchy convolution of $r$-periodic functions $f$ and $g$ is given by
\begin{equation} \label{Cauchy_periodic}
(f\otimes g)(n)= \sum_{k \ppmod r} f(n-k)g(k).
\end{equation}
It is well known that the Cauchy convolution $f\otimes g$ of two $r$-even functions $f$ and $g$ is also
$r$-even and
\begin{equation} \label{Cauchy_per}
\widehat{(f\otimes g)}(n)= \widehat{f}(n) \widehat{g}(n). 
\end{equation}
Thus 
\begin{equation} \label{Cauchy2}
(f\otimes g)(n)= \frac{1}{r} \sum_{d\mid r} \widehat{f}(r/d) \widehat{g}(r/d) c_d(n).
\end{equation}

\begin{example} \label{ex:DFT:delta} 
Let $\overline{\delta}_{r}$ denote the characteristic function of the regular integers modulo $r$, 
that is, $\overline{\delta}_{r}(n)=1$ if $n$ is regular modulo $r$ (i.e. if $(n, r)\Vert r$, see \cite{Tot2008}),
and 
$\overline{\delta}_{r}(n)=0$ otherwise. 
By the definition $(\ref{DFT})$ of DFT it is easy to see that 
\[
\widehat{\overline{\delta}}_r(n)  = \overline{c}_{r}(n). 
\]
Therefore 
\[
\widehat{\overline{c}}_r(n)  = r \overline{\delta}_{r}(n). 
\]
Let $\delta_r$ denote the Kronecker function, that is,
$\delta_r(n)=1$ if $(n,r)=1$, and $\delta_r(n)=0$ otherwise.
Then
\[
\widehat{\delta}_r(n)=c_r(n), \quad \widehat{c}_r(n)=r\delta_r(n). 
\]
These properties of Kronecker's function and Ramanujan's sum are well known 
\cite{H2001,S2008}. 
\end{example}

\begin{example}
% C6
Let $\overline{N}_{r, u}(n)$ denote the number of solutions of the congruence
\[
n 
\equiv x_1 + \cdots + x_u \pmod r
\]
such that each $x_i$ is regular {\rm (mod $r$)}. 
Then 
\[
\overline{N}_{r, u} = \overline{\delta}_{r}\otimes \cdots \otimes \overline{\delta}_{r}\quad  
(\overline{\delta}_{r}, u\ {\rm times})
\]
and
\[
\widehat{\overline{N}}_{r, u} = \Bigl(\widehat{\overline{\delta}}_{r}\Bigr)^u = \big(\overline c_r\big)^u. 
\]
We obtain 
\[ 
\overline{N}_{r,u}(n)= \frac1{r}\sum_{d\mid r} (\overline{c}_r(r/d))^u c_d(n).
\]
The function $\overline{N}_{r, 2}$ may be considered an analogue of Nagell's function $\theta_r$ under regular
integers and therefore could be denoted as $\overline{N}_{r, 2}=\overline{\theta}_r$. 
Thus
\[
\widehat{\overline{\theta}}_r= (\overline{c}_r)^2, 
\]
\[
\widehat{(\overline{c}_r)^2} = r\overline{\theta}_r. 
\]
The analogous results involving Ramanujan's sum are as follows \cite{H1997,H2001,M,MN,Sp}. 
Let $N_{r, u}(n)$ denote the number of
solutions of the congruence
\[
n \equiv x_1 + \cdots + x_u \pmod r
\]
with $(x_k,r) = 1$, $k=1,2,\ldots,u$. 
Then 
\[
N_{r, u} = \delta_r \otimes \cdots \otimes \delta_r\ (\delta_r, u\ {\rm times})
\]
and
\[
\widehat{N}_{r, u}
= \Bigl(\widehat\delta_r\Bigr)^u
= (c_r)^u. 
\]
% C3
Nagell's totient is given as $\theta_r = N_{r, 2}$.
Thus
\[
\widehat \theta_r= (c_r)^2, 
\]
\[
\widehat{(c_r)^2}= r\theta_r. 
\]

\end{example}

\begin{example} 
If $f$ and $g$ are even (mod $r$), then $(\ref{Cauchy2})$ can be written as  
\[
\sum_{a \ppmod r} f(n-a) g(a) 
= \frac{1}{r}\sum_{d | r}  \widehat{f}(r/d) \widehat{g}(r/d) \ c_d(n). 
\]
Let $g=\overline{\delta}_r$. 
Then $\widehat{g}=\overline c_r$ and therefore
\[
\sum_{
    \substack{a \ppmod r\\
              a\ {\rm regular}\ppmod r}}
     f(n-a)
=
 \frac{1}{r} \sum_{d | r} \widehat{f}(r/d) \overline c_r(r/d) c_d(n). 
\]
In particular, if $f(n)=c_r(n)$, then $\widehat{f}(n)=r\delta_r(n)$ and therefore 
\[
\sum_{
    \substack{a \ppmod r\\
              a\ {\rm regular}\ppmod r}}
     c_r(n-a)
=
 \overline c_r(1) c_r(n)
=
\overline{\mu}(r) c_r(n),  
\]
for $\overline{\mu}$, see Remark \ref{re:squareful}. 
The analogous results for the sums over reduced residue systems are \cite{H1997,M}
\[
\sum_{
    \substack{a \ppmod r\\
              (a, r)=1}}
     f(n-a)
=
 \frac{1}{r} \sum_{d | r} \widehat{f}(r/d)  c_r(r/d) c_d(n),  
\]
\[
\sum_{
    \substack{a \ppmod r\\
              (a, r)=1}}
     c_r(n-a)
= c_r(1) c_r(n)
= \mu(r) c_r(n). 
\]
\end{example}

\begin{theorem}\label{th:conv-r}
\[
\overline{c}_r(n) = \sum_{d \| r} c_d(n).
\]
\end{theorem}

\begin{proof}
Taking $f=\overline{c}_r$ in $(\ref{even_func})$ and using Example \ref{ex:DFT:delta} 
we obtain 
$$
\overline{c}_r(n)= \frac1{r} \sum_{d\mid r} r \overline{\delta}_r(r/d) c_d(n)
= \sum_{r/d\| r}c_d(n),   
$$
which proves Theorem \ref{th:conv-r}. 
This follows also from the definition of $\overline{c}_r(n)$ by grouping the terms according to the values
$(a, r)=d\Vert r$.
\end{proof}

\begin{remark}\label{re:squareful}
For $n=0$ we have $\overline{c}_r(0)=\varrho(r)$.  
%which can be written as $\varrho(r)=\sum_{d\| r} \phi(d)$. 
For $n=1$ we
obtain an analogue of the M\"obius function as 
\[
\overline{c}_r(1) = \sum_{d \| r} \mu(d) \equiv \overline{\mu}(r), 
\]
which is the characteristic function of the squareful positive integers. 
For the usual Ramanujan's sum we have $c_r(0)=\phi(r)$ and $c_r(1)=\mu(r)$. 
\end{remark}

\begin{remark}
Theorem \ref{th:conv-r} gives a convolutional expression of 
$\overline{c}_r(n)$ with respect to the variable $r$. 
In fact, Theorem \ref{th:conv-r} can be written as 
$$
\overline{c}_{({\bf\cdot})}(n) = c_{({\bf\cdot})}(n)\oplus 1. 
$$
For $n=0$ we obtain the known relation 
$$
\varrho = \phi\oplus 1. 
$$
\end{remark}

As a consequence of Theorem \ref{th:conv-r} we obtain another 
convolutional expression of 
$\overline{c}_r(n)$ with respect to the variable $r$. 

\begin{theorem}  \label{th:conv-r2}
\begin{equation*}
\overline{c}_r(n)= \sum_{de=r} \mu(d)\sigma^*_{(d)}((n,e)_*),
\end{equation*}
where $\sigma^*_{(k)}(m)$ is the sum of those unitary divisors $a$ of
$m$ for which the complementary divisor $m/a$ is prime to $k$ and where 
$(n, m)_*=\max \{d: d\mid n, d\Vert m \}$.
\end{theorem}

\begin{proof} From Theorem \ref{th:conv-r} we obtain 
\begin{equation*}
\overline{c}_r(n) = \sum_{d\Vert r} c_d(n) = \sum_{d\mid \mid r}
\sum_{e\mid (n,d)} e\mu(d/e).
\end{equation*}
Let $dj=r, (d,j)=1$, $ey=d$ and obtain
\begin{equation*}
\overline{c}_r(n) = \sum_{\substack{eyj=r\\(ey,j)=1\\e\mid n}} e\mu(y)
= \sum_{yk=r} \mu(y) \sum_{\substack{ej=k\\ (e,j)=1 \\ (y,j)=1\\e\mid n}} e
\end{equation*}
\begin{equation*}
=\sum_{yk=r} \mu(y) \sum_{\substack{e\mid \mid (n,k)_* \\
(k/e,y)=1}} e = \sum_{yk=r} \mu(y) \sigma^*_{(y)}((n,k)_*).
\end{equation*}
\end{proof}

Taking $n=r$ in Theorem \ref{th:conv-r2} we obtain a new convolutional expression for $\varrho(r)$. 

\begin{corollary}  
\begin{equation*}
\varrho(r)= \sum_{de=r} \mu(d) \sigma^*_{(d)}(e).
\end{equation*}
\end{corollary}

\begin{remark}
Applying Theorem \ref{th:conv-r} and known results on Ramanujan's sum we can derive 
results for $\overline{c}_r(n)$. As examples we present the following. 
It is well known \cite{M} that for $s, t\mid r$, 
\begin{equation}\label{eq:ccn}
\sum_{a+b\equiv n\ ({\rm mod}\ {r})} c_s(a) c_t(b)
= \begin{cases}  
  rc_s(n) & \text{if}\ s=t,\\
  0 \ & \text{otherwise}.
\end{cases}
\end{equation}
Taking $n=0$ we obtain for $s, t\mid r$, 
\begin{equation}\label{eq:cc0}
\sum_{a\ ({\rm mod}\ {r})} c_s(a) c_t(a)
= \begin{cases}  
  r\phi(s) & \text{if}\ s=t,\\
  0 \ & \text{otherwise}.
\end{cases}
\end{equation}
Analogous results for $\overline{c}_r(n)$ are 
\begin{equation}\label{eq:-c-cn}
\sum_{a+b\equiv n\ ({\rm mod}\ {r})} \overline{c}_s(a) \overline{c}_t(b)
= r [c_{({\bf\cdot})}(n)\oplus 1]((s, t)_{**}),
\end{equation}
\begin{equation}\label{eq:-c-c0}
\sum_{a\ ({\rm mod}\ {r})} \overline{c}_s(a) \overline{c}_t(a)
= r \varrho((s, t)_{**}),
\end{equation}
\begin{equation}\label{eq:-ccn}
\sum_{a+b\equiv n\ ({\rm mod}\ {r})} \overline{c}_s(a) c_t(b)
= \begin{cases}  
  rc_t(n) & \text{if}\ t\Vert s,\\
  0 \ & \text{otherwise}, 
\end{cases}
\end{equation}
\begin{equation}\label{eq:-cc0}
\sum_{a\ ({\rm mod}\ {r})} \overline{c}_s(a) c_t(a)
= \begin{cases}  
  r\phi(t) & \text{if}\ t\Vert s,\\
  0 \ & \text{otherwise}, 
\end{cases}
\end{equation}
where $s, t\mid r$ and $(s, t)_{**}$ stands for the greatest common unitary divisor of $s$ and
$t$. 
In order to prove $(\ref{eq:-c-cn})$ we apply Theorem \ref{th:conv-r} to obtain 
\[
\sum_{a+b\equiv n\ ({\rm mod}\ {r})} \overline{c}_s(a) \overline{c}_t(b)
= \sum_{a+b\equiv n\ ({\rm mod}\ {r})} 
   \sum_{d \| s} c_d(a) \sum_{e \| t} c_e(b)
= \sum_{d \| s} \sum_{e \| t} 
    \sum_{a+b\equiv n\ ({\rm mod}\ {r})} c_d(a) c_e(b). 
\]
Now, applying $(\ref{eq:ccn})$ we get $(\ref{eq:-c-cn})$. Taking $n=0$ in $(\ref{eq:-c-cn})$ gives
$(\ref{eq:-c-c0})$.  
Equations $(\ref{eq:-ccn})$ and $(\ref{eq:-cc0})$ follows in a similar way. 
Equations $(\ref{eq:-c-cn})$ and $(\ref{eq:-ccn})$ could also be proved applying an analogue of
$(\ref{Cauchy2})$,  see \cite[Eq. (4.1)]{H1997}.  
\end{remark}

%%%%%%%%%%%%%%%%%%%%%%%%%%%%%%%%%%%%%%%%%%%%%%%%%%%

\section{Applications of the Parseval formula}

We present an analogue of the Parseval formula \cite{MV2007} in harmonic analysis
and its consequences.
In fact, 
if  $f$ is $r$-periodic, then 
\begin{equation}
\sum_{n=1}^r |f(n)|^2 
= \frac{1}{r}\sum_{k=1}^r |\widehat{f}(k)|^2.
\end{equation}

\begin{example}
% I2
Let $f=\overline c_r$. Then 
\begin{equation*}
\sum_{n=1}^r \big(\overline c_r(n)\big)^2
=\frac{1}{r} \sum_{n=1}^r r^2(\overline{\delta}_{r}(n))^2 
=r \sum_{n=1\atop n\ {\rm regular}}^r 1
=r\, \varrho(r). 
\end{equation*} 
For $f=\overline{N}_{r, u}$ we have 
\[
\sum_{n=1}^r \big( \overline{N}_{r, u}(n) \big)^2 
= \frac1 r \sum_{n=1}^r \big(\overline c_r(n)\big)^{2u}. 
\]
For $u=2$ i.e. for $f=\overline{\theta}_r$ we have 
\[
\sum_{n=1}^r \big(\overline{\theta}_r(n)\big)^2 
= \frac1 r \sum_{n=1}^r \big(\overline{c}_r(n)\big)^4. 
\]
Analogous results for Ramanujan's sum are 
\[
\sum_{n=1}^r (c_r(n))^2 
= \frac1{r} \sum_{n=1}^r r^2 \big(\delta_r(n)\big)^2 
= r \sum_{\substack{n=1\\
(n, r)=1}}^r 1 =r \phi(r),
\]
\[
\sum_{n=1}^r \big( N_{r, u}(n) \big)^2 
= \frac1 r \sum_{n=1}^r \big(c_r(n)\big)^{2u},  
\]
\[
\sum_{n=1}^r \big(\theta_r(n)\big)^2 
= \frac1 r \sum_{n=1}^r \big(c_r(n)\big)^4. 
\]

\end{example}

%%%%%%%%%%%%%%%%%%%%%%%%%%%%%%%%%%%%%%%%%%%%%%%%%%%%

\section{Multiplicative functions}\label{Selberg}

An arithmetical function $f\colon\N^u\to\C$ of $u$ 
variables is said to be Selberg multiplicative
in $n_1, n_2,\ldots, n_u$ if
for each prime $p$ there exists $F_p\colon\N_0^u\to\C$ 
with $F_p(0, 0,\ldots, 0)=1$ for
all but finitely many $p$ such that
\begin{equation}\label{e:Sel-exp}
f(n_1, n_2,\ldots, n_u)=\prod_{p} F_p(\nu_p(n_1), 
\nu_p(n_2),\ldots, \nu_p(n_u))
\end{equation}
for all $n_1, n_2,\ldots, n_u\in\N$.
Here $\nu_p(n)$ is the exponent of $p$ in the canonical factorization of $n$. 
Note that if $f$ is Selberg multiplicative
in $n_1, n_2,\ldots, n_u$, then $f$ is Selberg multiplicative
in any nonempty subset of the set of variables $n_1, n_2,\ldots, n_u$. 

An arithmetical function $f\colon\N^u\to\C$ of $u$ 
variables is said to be multiplicative
in $n_1, n_2,\ldots, n_u$ if $f(1, 1,\ldots, 1)=1$ and 
\begin{equation}
f(n_1, n_2,\ldots, n_u)f(m_1, m_2,\ldots, m_u)
=f(n_1m_1, n_2m_2,\ldots, n_um_u)
\end{equation}
for all $n_1, n_2,\ldots, n_u\in\N$ and $m_1, m_2,\ldots, m_u\in\N$ 
with $(n_1 n_2\cdots n_u, m_1m_2\cdots m_u)=1$. 
Multiplicative functions $f$ are Selberg multiplicative with 
\begin{equation}
F_p(\nu_p(n_1), \nu_p(n_2),\ldots, \nu_p(n_u))=f(p^{\nu_p(n_1)}, 
p^{\nu_p(n_2)},\ldots, p^{\nu_p(n_u)}). 
\end{equation}

An arithmetical function $f\colon\N\to\C$ (of one variable) is said to be
semimultiplicative if
\begin{equation}
f(m)f(n)=f((m, n))f([m, n])
\end{equation}
for all $m, n\in\N$, where $(m, n)$ and $[m, n]$ stand for the gcd
and lcm of $m$ and $n$. It is known that  an arithmetical function
$f$ (not identically zero) is semimultiplicative if and only if
there exists a nonzero constant $c_f$, a positive integer $a_f$ and
a multiplicative function $f_m$ such that
\begin{equation}\label{e:sem-rea}
f(n)=c_f f_m(n/a_f)
\end{equation}
for all $n\in\N$. (We interpret that the arithmetical function $f_m$
possesses the property that $f_m(x)=0$ if $x$ is not a positive
integer.) Note that $c_f=f(a_f)$. We will take $a_f$ as the smallest
positive integer $k$ such that $f(k)\ne 0$. Further, it is known
that an arithmetical function (of one variable) is Selberg multiplicative if and only
if it is semimultiplicative. 
A semimultiplicative function $f\colon\N\to\C$ not identically zero possesses a
Selberg expansion as
\begin{equation}
f(n)=f(a_f)\prod_{p}
\left(\frac{f(a_fp^{\nu_p(n)-\nu_p(a_f)})}{f(a_f)}\right).
\end{equation}
Semimultiplicative functions $f$ with
$f(1)\ne 0$ are known as quasimultiplicative functions.
Quasimultiplicative functions $f$ possess the property
$f(1)f(mn)=f(m)f(n)$ whenever $(m, n)=1$. Semimultiplicative
functions $f$ with $f(1)=1$ are the usual multiplicative functions.

For material on Selberg multiplicative and
semimultiplicative functions we refer to
\cite{BHS,Hau,HHS,Rearick66a,Rearick66b,Selberg,Si}. 

We here present some multiplicative properties of 
the analogue of Ramanujan's sum with respect to regular integers (mod $r$). 
Similar results for the usual Ramanujan's sum are given at the end of this section. 
We begin with a general theorem on multiplicative $r$-even functions. 

\begin{theorem}\label{th:gen-mult}
Let $f(n, r)$ be an arithmetical function of two variables. 
If for each $r\ge 1$, $f(n, r)$ is $r$-even as a function of $n$ and 
if for each $n\ge 1$, $f(n, r)$ is multiplicative in $r$, 
then $f(n, r)$ is multiplicative as a function of two variables $r$ and $n$ $(\in\N)$. 
\end{theorem}

\begin{proof} 
Let $(mr, ns)=1$. Then 
\begin{eqnarray*}
f(mn, rs) 
&=&  f(mn, r) f(mn, s) = f((mn, r), r)f((mn, s), s)\\
&=&  f((m, r), r) f((n, s), s)= f(m, r)f(n, s). 
\end{eqnarray*}
The proof is completed. 
\end{proof}

\begin{theorem}\label{th:mult}
{\rm  (1)} For each $n\in\Z$, $\overline{c}_r(n)$ is multiplicative in $r$.

{\rm  (2)} For each $r\in\N$, $\overline{c}_r(n)$ is semimultiplicative in $n$ $(\in\N)$. 

{\rm  (3)} $\overline{c}_r(n)$ is multiplicative
as a function of two variables $r$ and $n$ $(\in\N)$. 
\end{theorem}

\begin{proof}
(1) According to Theorem \ref{th:conv-r} we know that $\overline{c}_r(n)$ in $r$ 
is the unitary convolution of Ramanujan's sum $c_r(n)$ in $r$ and the constant 
function $1(r)$. Since these functions are multiplicative in $r$ and the unitary convolution of 
two multiplicative functions is multiplicative \cite{M}, we see that 
$\overline{c}_r(n)$ is multiplicative in $r$. 

(2) We utilize Remark \ref{re:conv-n}. The functions $\eta_{r}(n)$ and $1(n)$ are multiplicative in $n$ 
and therefore they are  semimultiplicative in $n$. We show that the function 
$\overline{\mu}_{r}\big(r/n\big)$ is semimultiplicative in $n$. 
In fact, 
$$
\overline{\mu}_{r}\big(r/n\big)
=\prod_{p} \overline{\mu}_{r}\big(p^{\nu_p(r)-\nu_p(n)}\big),
$$ 
where $\overline{\mu}_{r}(x)=0$ if $x$ is not a positive integer, that is, 
$\overline{\mu}_{r}\big(p^{\nu_p(r)-\nu_p(n)}\big)=0$ if $\nu_p(r)-\nu_p(n)<0$. 
This shows that $\overline{\mu}_{r}\big(r/n\big)$ is Selberg multiplicative in $n$, 
which means that $\overline{\mu}_{r}\big(r/n\big)$ is semimultiplicative in $n$. 
The usual product  and the Dirichlet convolution of semimultiplicative functions is 
semimultiplicative \cite{Rearick66a}. 
Therefore $\overline{c}_r(n)$ is semimultiplicative in $n$. 

(3) This follows directly from Theorem \ref{th:gen-mult}. 
\end{proof}

\begin{remark}
We know that $\overline{c}_r(n)$ is multiplicative in $r$.  
Its values at prime powers $r=p^k$ are given as 
$$
\overline{c}_{p^k}(n)=1+c_{p^k}(n)
=\begin{cases}
1+p^k-p^{k-1} & \text{if $p^k\mid n$}\\
   1-p^{k-1} & \text{if $p^{k-1}\mid n, p^k\nmid n$}\\
          1 & \text{otherwise.} 
\end{cases} 
$$
Note that for $n=0$ we have $\varrho(p^k)=1+\phi(p^k)$. 
\end{remark}

\begin{remark}\label{re:an-ram-mult}
We know that $\overline{c}_r(n)$ is semimultiplicative in $n$ $(\in\N)$. 
The smallest value of $n$ $(\in\N)$ for which $\overline{c}_r(n)\ne 0$ is $n=\prod_{p\| r} p$. 
In fact, by multiplicativity in $r$ we have 
$$
\overline{c}_r(n)=\prod_{p^k\| r} (1+c_{p^k}(n)). 
$$
If $k=1$ (i.e. $p\| r$), then $1+c_{p^k}(1)=1+\mu(p)=0$ and 
$1+c_{p^k}(p)=1+\phi(p)=p\ne 0$. 
If $k\ge 2$, then $1+c_{p^k}(1)=1+\mu(p^k)=1\ne 0$. 
This shows that 
$$
n=\prod_{p\| r} p \prod_{p^k\| r\atop k\ge 2} 1
$$ 
 is 
the smallest value of $n$ $(\in\N)$ for which $\overline{c}_r(n)\ne 0$. 
The function value of $\overline{c}_r(n)$ at $n=\prod_{p\| r} p$ is also $\prod_{p\| r} p$. 
This implies that $\overline{c}_r(n)$ is multiplicative in 
$n$ if and only if  $r$ is squareful  or $r=1$. 

We could also show that $\overline{\mu}_{r}\big(r/n\big)$ 
is semimultiplicative in $n$ such that $n=\prod_{p\| r} p$ is 
the smallest value of $n$ for which $\overline{\mu}_{r}\big(r/n\big)\ne 0$, 
and then obtain properties of $\overline{c}_r(n)$ in $n$ 
using Remark \ref{re:conv-n}. 
\end{remark}

\begin{remark}
Ramanujan's sum $c_r(n)$ is multiplicative
as a function of two variables $r$ and $n$ $(\in\N)$, multiplicative in $r$ and  
semimultiplicative in $n$ $(\in\N)$, see \cite{AA,Hau,SH}.  
The smallest value of $n$ $(\in\N)$ for which $c_r(n)\ne 0$ is $n=r/\gamma(r)$,
where $\gamma(r)$ is the product of the distinct 
prime factors of $r$. 
This follows from the property 
\begin{equation}\label{eq:ram-values}
c_{p^k}(n)
=\begin{cases}
p^k-p^{k-1} & \text{if $p^k\mid n$}\\
-p^{k-1} & \text{if $p^{k-1}\mid n, p^k\nmid n$}\\
0 & \text{otherwise.} 
\end{cases} 
\end{equation}
Ramanujan's sum $c_r(n)$ is 
quasimultiplicative in $n$ $(\in\N)$ if and only if $r$ is squarefree.
It is multiplicative in $n$ $(\in\N)$ 
if and only if $r$ is squarefree and
$\omega(r)$ is even, where $\omega(r)$ is the number of distinct prime 
factors of $r$. This follows from the property \cite[p. 90]{M}
\begin{equation}\label{eq:ram-quasi}
c_r(m)c_r(n)=\mu(r)c_r(mn)\ \ {\rm if}\ (m, n)=1.  
\end{equation}
\end{remark}

\begin{remark} It is known \cite{Si} that if $f(n, r)$ is multiplicative as a function of two
variables, then for any $r\ge 1$, $f(m, r)f(n, r)= f(1, r) f(mn, r)$ whenever $(m, n)=1$. 
Taking $f(n, r)=c_r(n)$ gives $(\ref{eq:ram-quasi})$ and taking $f(n, r)=\overline{c}_r(n)$ 
gives an analogue of $(\ref{eq:ram-quasi})$ as 
\begin{equation}\label{eq:an-ram-quasi}
\overline{c}_r(m) \overline{c}_r(n)=\overline{\mu}(r) \overline{c}_r(mn)\ \ {\rm if}\ (m, n)=1.  
\end{equation}
This implies that $\overline{c}_r(n)$ is multiplicative in 
$n$ if and only if  $r$ is squareful  or $r=1$, see also Remark \ref{re:an-ram-mult}. 
\end{remark}

%%%%%%%%%%%%%%%%%%%%%%%%%%%%%%%%%%%%%%%%%%%%%%

\section{Mean values}

The mean value of an arithmetical function $f$ is 
$m(f)=\lim_{x\to\infty} \frac1{x}\sum_{n\le x} f(n)$ if this limit exists. 
(Here $n$ runs through the integers from $1$ to $[x]$.) 
It is known that $\sum_{n\le x} c_r(n)= O(1)$ for $r>1$. 
Thus the mean value of the function $c_r(\cdot)$ is $0$ for $r>1$, see \cite{T}. 
We here present an analogous result for $\overline{c}_r(\cdot)$.

\begin{theorem} \label{th:mean}
For any $r\ge 1$, 
\begin{equation*}
\sum_{n\le x} \overline{c}_r(n) = x + O(1).
\end{equation*}

\end{theorem}

\begin{proof}
Applying  Theorem \ref{th:conv-n} we have 
\begin{align*}
    \sum_{n\le x} \overline{c}_r(n)
    &= \sum_{n\le x} \sum_{d\mid (n, r)}d \overline{\mu}_r(r/d)
    = \sum_{d\mid r} d \overline{\mu}_r(r/d) [x/d]\\
    &= x \sum_{d\mid r} \overline{\mu}_r(r/d)-\sum_{d\mid r} d\overline{\mu}_r(r/d)\{x/d\}. 
\end{align*}
Since $\sum_{d\mid r}\overline{\mu}_r(r/d)=g_r(r)=1$ and $0\le\{x/d\}<1$, we obtain the result. 
\end{proof}

\begin{corollary} 
The mean value of the function $\overline{c}_r(\cdot)$ is $1$.
\end{corollary}

\begin{remark}
Note that 
$$
\sum_{n\le x} \overline{c}_r(n)=g_r(r) x+O(1), 
$$
$$
\sum_{n\le x} c_r(n)=\delta_r(r) x+O(1). 
$$
Since $\delta_r(r)=0$ for $r>1$, we have $\sum_{n\le x} c_r(n)= O(1)$ for $r>1$. 
For $r=1$, we have $\sum_{n\le x} c_r(n)= x+ O(1)$. To be more precise, 
$\sum_{n\le x} c_1(n)=[x]$. 
\end{remark}

\begin{remark}
From the proof of Theorem \ref{th:mean} we obtain 
$$
\sum_{n\le x} \overline{c}_r(n)=g_r(r) x = x\ \ {\rm for}\ \ r\mid x\in \N. 
$$
In a similar way  
$$
\sum_{n\le x} c_r(n)=\delta_r(r) x = 0\ \ {\rm for}\ 1<r\mid x\in \N,  
$$
see \cite{A1972}. 
\end{remark}

%%%%%%%%%%%%%%%%%%%%%%%%%%%%%%%%%%%%%%%%%%%%%%

\section{Dirichlet series}

For Ramanujan's sum we have  
\begin{equation}
\sum_{n=1}^\infty
\frac{c_r(n)}{n^s} = \phi_{1-s}(r)\zeta(s) 
\end{equation}
for $\RE(s)>1$, where  $\phi_{t}(r)=\sum_{d\mid r} d^{t} \mu(r/d)$, see \cite{M,Y}.  
A similar result holds for the analogue of Ramanujan's sum with respect to 
regular integers (mod $r$).

\begin{theorem} \label{th:DS-n}
For any $r\ge 1$ and $\RE(s)>1$, 
\begin{equation*}
\sum_{n=1}^\infty
\frac{\overline{c}_r(n)}{n^s} = (\phi_{1-s}\oplus 1)(r)\zeta(s). 
\end{equation*}
\end{theorem}

\begin{proof}
The function $\overline{c}_r(n)$ is bounded for any $r\ge 1$ 
(in fact, $|\overline{c}_r(n)|\le \sigma(r)$ by Theorem \ref{th:conv-n});  
hence the series $\sum_{n=1}^\infty\frac{\overline{c}_r(n)}{n^s}$ 
converges absolutely for $\RE(s)>1$. 
Therefore we can change the order of summation as 
\begin{eqnarray*}
\sum_{n=1}^\infty \frac{\overline{c}_r(n)}{n^s}
&=& \sum_{n=1}^\infty \frac{1}{n^s} \sum_{d\mid n\atop d\mid r} d \overline{\mu}_r(r/d)
 =  \sum_{d\mid r} d \overline{\mu}_r(r/d) \sum_{e=1}^\infty  \frac{1}{d^s e^s}\\
&=& \left(\sum_{d\mid r} d^{1-s} \overline{\mu}_r(r/d)\right) \zeta(s). 
\end{eqnarray*}
On the basis of Lemma \ref{le:mu-r} we have 
$$
\sum_{d\mid r} d^{1-s} \overline{\mu}_r(r/d)=(\phi_{1-s}\oplus 1)(r). 
$$
This completes the proof. 
\end{proof}

\medskip

Next we consider Ramanujan's formula
\begin{equation}
\sum_{n=1}^{\infty} \frac{c_r(n)}{n}= -\Lambda(r), \quad r>1,
\end{equation}
where $\Lambda$ is the von Mangoldt function, see \cite{A1972}.

\begin{theorem} \label{th:vonM}
For any $r\ge 1$, 
\begin{equation*}
\sum_{n\le x} 
\frac{\overline{c}_r(n)}{n} 
= \log x+C-(\Lambda\oplus 1)(r)+ O(x^{-1}),
\end{equation*}
where $C$ is Euler's constant.
\end{theorem}

\begin{proof}
Applying  Theorem \ref{th:conv-n} we have  
\begin{align*}
    \sum_{n\le x} \frac{\overline{c}_r(n)}{n}
    &= \sum_{n\le x} \frac1{n} \sum_{d\mid (n,r)} d \overline{\mu}_r(r/d)
    = \sum_{d\mid r} \overline{\mu}_r(r/d) \sum_{j\le x/d} \frac1{j} \\
    &= \sum_{d\mid r} \overline{\mu}_r(r/d) \Bigl( \log(x/d)+ C + O(d/x)\Bigr) \\
    &= (\log x+C)\sum_{d\mid r} \overline{\mu}_r(r/d) - \sum_{d\mid r} \overline{\mu}_r(r/d) \log d
        + O\biggl(x^{-1} \sum_{d\mid r} d|\overline{\mu}_r(r/d)|\biggr).  
\end{align*}
Applying Lemma \ref{le:mu-r} we obtain 
$$
\sum_{n\le x} \frac{\overline{c}_r(n)}{n}
= (\log x+C) g_r(r) - [(\log *\mu)\oplus 1](r) + O(x^{-1}).
$$
Since $g_r(r)=1$ and $\log\ast\mu = \Lambda$, 
we obtain Theorem \ref{th:vonM}. 
\end{proof}

\begin{remark}
In Theorem \ref{th:vonM}, $(\Lambda\oplus 1)(r)=\log\gamma(r)$, where $\gamma(r)=\prod_{p\mid r} p$. 
\end{remark}

\begin{remark}
It follows from Theorem \ref{th:vonM} that $\sum_{n\le x} \frac{\overline{c}_r(n)}{n}$ 
tends to infinity as $x\to\infty$. 
\end{remark}

\begin{remark}
Note that the leading term of  $\sum_{n\le x} \frac{\overline{c}_r(n)}{n}$ is 
$g_r(r)(\log x+C)$, while that of \break $\sum_{n\le x} \frac{c_r(n)}{n}$ is 
$\delta_r(r)(\log x+C)$. 
Since $\delta_r(r)=0$ for $r>1$, the series $\sum_{n=1}^{\infty} \frac{c_r(n)}{n}$ converges for $r>1$, 
and since $g_r(r)=1$ for $r\ge 1$, 
the series $\sum_{n=1}^{\infty} \frac{\overline{c}_r(n)}{n}$ diverges for $r\ge 1$. 
In a similar way, the series  
$\sum_{n=1}^{\infty} \frac{c_r(n)}{n}(=\sum_{n=1}^{\infty} \frac{1}{n})$ diverges for $r=1$. 
\end{remark}

\begin{remark}
Theorems \ref{th:mean}, \ref{th:DS-n} and \ref{th:vonM} can also be proved applying 
Theorem \ref{th:conv-r} and corresponding results for Ramanujan's sum $c_r(n)$. 
 
\end{remark}

\medskip

We next present an analogue of the formula
\begin{equation}
\sum_{r=1}^\infty \frac{c_r(n)}{r^s} = \frac{1}{\zeta(s)}\frac{\sigma_{s-1}(n)}{n^{s-1}}
\end{equation}
for $n\ge 1$ and $\RE(s)>1$, where $\sigma_{s-1}(n)=\sum_{d\mid n} d^{s-1}$, see \cite{M}. 

\begin{theorem} \label{th:DS-r}
For $n\ge 1$ and $\RE(s)>1$,
\begin{equation*}
\sum_{r=1}^{\infty} \frac{\overline{c}_r(n)}{r^s} =
\end{equation*}
\begin{equation*}
=\frac{\zeta(2s)\zeta(3s)}{\zeta(6s)} \cdot
\frac{\phi_{2s}(n)\phi_{3s}(n)}{\phi_{6s}(n)}\cdot
\frac{n^{s+1}\phi_{s-1}(n)}{\phi_s(n)} \cdot \prod_{p^a\mid\mid n}
\left(1 - \frac1{p^{s-1}}\left(1-\left(1-\frac1{p^s}
\right)^2\left(1-\frac1{p^{a(s-1)}} \right) \right) \right).
\end{equation*}
\end{theorem}

\begin{proof} It can be shown that $\sum_{r=1}^\infty\frac{\overline{c}_r(n)}{r^s}$ 
converges absolutely for $\RE(s)>1$. 
Therefore we can change the order of summation in this series. 
Applying  Theorem \ref{th:conv-r} we have  
\begin{equation*}
\sum_{r=1}^{\infty} \frac{\overline{c}_r(n)}{r^s} =
\sum_{r=1}^{\infty} \frac1{r^s} \sum_{d\mid\mid r} c_d(n) =
\sum_{d=1}^{\infty} \frac{c_d(n)}{d^s} 
\sum_{\substack{e=1\\(e,d)=1}}^{\infty} \frac1{e^s}. 
\end{equation*}
It is known \cite{Cohen61} that 
$$
\sum_{\substack{e=1\\(e,d)=1}}^{\infty} \frac1{e^s}=\zeta(s)\frac{\phi_s(d)}{d^{s}}. 
$$
Therefore
\begin{equation*}
\sum_{r=1}^{\infty} \frac{\overline{c}_r(n)}{r^s} 
= \zeta(s) \sum_{d=1}^{\infty} \frac{\phi_s(d)c_d(n)}{d^{2s}}
= \zeta(s) \prod_p \left(1+\left(1-\frac1{p^s}\right)
\sum_{k=1}^{\infty}
\frac{c_{p^k}(n)}{p^{ks}}\right). 
\end{equation*}
Applying formula $(\ref{eq:ram-values})$ for the values $c_{p^k}(n)$ we obtain Theorem
\ref{th:DS-r}.
\end{proof}

\begin{remark}
For $n=1$ we obtain the well-known formula
\begin{equation*}
\sum_{r=1}^{\infty} \frac{\overline{\mu}(r)}{r^s}
=\frac{\zeta(2s)\zeta(3s)}{\zeta(6s)},
\end{equation*}
where $\overline{c}_r(1)=\overline{\mu}(r)$ is the characteristic function
of the squareful integers.
\end{remark}

Finally we present the Dirichlet series of the generalized M\"{o}bius function $\overline{\mu}_r(n)$. 
Let $\phi_s^{\ast}(r)$ denote the unitary analogue of Jordan's totient, that is, 
\begin{equation}
\phi_{s}^{\ast}(r)=\sum_{d\| r} d^{s}\mu^{\ast}(r/d), 
\end{equation}
where $\mu^{\ast}$ is the unitary analogue of the M\"{o}bius function \cite{Cohen60}. 
In other words, 
\begin{equation}\label{eq:un-Jor}
\phi_{s}^{\ast}(r)=r^s \prod_{p^a\| r}\left(1-\frac1{p^{as}}\right). 
\end{equation}

\begin{theorem} \label{th:mu-r-DS}
For any $r\ge 1$ and $\RE(s)>1$,
\begin{equation*}
\sum_{n=1}^{\infty} \frac{\overline{\mu}_r(n)}{n^s}
= 
\frac{1}{\zeta(s)} 
\frac{\phi_{2s}^{\ast}(r)}{r^s \phi_{s}^{\ast}(r)}.
\end{equation*}
\end{theorem}

\begin{proof} Applying $(\ref{eq:mu-r})$ we have 
\begin{equation*}
\sum_{n=1}^{\infty} \frac{\overline{\mu}_r(n)}{n^s}=
\frac{1}{\zeta(s)} \prod_{p^a\| r}\left(1+\frac1{p^{as}}\right).  
\end{equation*}
Now, 
Theorem \ref{th:mu-r-DS} follows by $(\ref{eq:un-Jor})$.  
\end{proof}

\begin{remark}
The Dirichlet series of the generalized M\"{o}bius function $\overline{\mu}_r(n)$ can also be written in terms
of the unitary analogues of generalized Dedekind's totient and divisor-sum function \cite{H2002} as 
\begin{equation*}
\sum_{n=1}^{\infty} \frac{\overline{\mu}_r(n)}{n^s}
= 
\frac{1}{\zeta(s)} 
\frac{\psi_{s}^{\ast}(r)}{r^s}
= 
\frac{1}{\zeta(s)} 
\frac{\sigma_{s}^{\ast}(r)}{r^s}, 
\end{equation*}
where 
\begin{equation*}
\psi_{s}^{\ast}(r)=r^s \prod_{p^a\| r}\left(1+\frac1{p^{as}}\right) 
\end{equation*}
and $\sigma_{s}^{\ast}(r)$ is the sum of the $s$th powers of the unitary divisors of $r$. 
Note that $\psi_{s}^{\ast}=\sigma_{s}^{\ast}$. 
\end{remark}

\medskip
\noindent
{\bf Acknowledgement}\ Financial support from The Magnus Ehrnrooth Foundation is greatfully 
acknowledged.

\medskip
\noindent
Pentti Haukkanen,\\ 
Department of Mathematics and Statistics\\
FI-33014 University of Tampere, Finland \\

\noindent
L\'aszl\'o T\'oth, \\
Institute of Mathematics and Informatics \\
University of P\'ecs \\
Ifj\'us\'ag u. 6 \\
7624 P\'ecs, Hungary \\

\end{document}